\pdfoutput=1
\RequirePackage{ifpdf}
\ifpdf 
\documentclass[pdftex]{sigma}
\else
\documentclass{sigma}
\fi

\usepackage{mathtools} 
\usepackage{slashed} 
\usepackage{mathrsfs}

\usepackage{tikz}
\usetikzlibrary{cd} 

\usepackage{rudismacros}

\numberwithin{equation}{section}

\newtheorem{Theorem}{Theorem}[section]
\newtheorem{Corollary}[Theorem]{Corollary}

\newtheorem{Proposition}[Theorem]{Proposition}
\newtheorem{Conjecture}[Theorem]{Conjecture}
 { \theoremstyle{definition}
\newtheorem{Definition}[Theorem]{Definition}

\newtheorem{Example}[Theorem]{Example}
\newtheorem{Remark}[Theorem]{Remark} }

\newcommand*{\Mishchenko}{Mi\v{s}\v{c}enko}

\begin{document}

\allowdisplaybreaks

\newcommand{\arXivNumber}{2008.13754}

\renewcommand{\thefootnote}{}

\renewcommand{\PaperNumber}{127}

\FirstPageHeading

\ShortArticleName{Width, Largeness and Index Theory}

\ArticleName{Width, Largeness and Index Theory\footnote{This paper is a~contribution to the Special Issue on Scalar and Ricci Curvature in honor of Misha Gromov on his 75th Birthday. The full collection is available at \href{https://www.emis.de/journals/SIGMA/Gromov.html}{https://www.emis.de/journals/SIGMA/Gromov.html}}}

\Author{Rudolf ZEIDLER}

\AuthorNameForHeading{R.~Zeidler}

\Address{Mathematical Institute, University of M\"unster, Einsteinstr.~62, 48149 M\"unster, Germany}
\Email{\href{mailto:math@rzeidler.eu}{math@rzeidler.eu}}
\URLaddress{\url{https://www.rzeidler.eu}}

\ArticleDates{Received September 01, 2020, in final form November 26, 2020; Published online December 02, 2020}

\Abstract{In this note, we review some recent developments related to metric aspects of scalar curvature from the point of view of index theory for Dirac operators. In particular, we revisit index-theoretic approaches to a conjecture of Gromov on the width of Riemannian bands $M \times [-1,1]$, and on a conjecture of Rosenberg and Stolz on the non-existence of complete positive scalar curvature metrics on $M \times \mathbb{R}$. We show that there is a more general geometric statement underlying both of them implying a quantitative negative upper bound on the infimum of the scalar curvature of a complete metric on $M \times \R$ if the scalar curvature is positive in some neighborhood. We study ($\hat{A}$-)iso-enlargeable spin manifolds and related notions of width for Riemannian manifolds from an index-theoretic point of view. Finally, we list some open problems arising in the interplay between index theory, largeness properties and width.}


\Classification{58J22; 19K56; 53C21; 53C23}

\renewcommand{\thefootnote}{\arabic{footnote}}
\setcounter{footnote}{0}

\section{Introduction}
Gromov proposed studying the geometry of scalar curvature via various metric inequalities which have similarities to classical Riemannian comparison geometry, see in particular~\cite{Gromov:MetricInequalitiesScalar,gromovFourLecturesScalar2019}.
One instance of such an inequality is the following estimate on the widths of Riemannian bands.
\begin{Conjecture}[{\cite[Section~11.12, Conjecture~C]{Gromov:MetricInequalitiesScalar}}]\label{conj:band-width}
 Let \(M\) be a closed manifold of dimension \(n-1\neq 4\) such that \(M\) does not admit a metric of positive scalar curvature.
 Then every Riemannian metric \(g\) on \(V = M \times [-1,1]\) of scalar curvature bounded below by \(n(n-1) = \scal_{\Sphere^n}\) satisfies
 \[
 \width(V, g) \coloneqq \dist_g(\partial_- V, \partial_+ V) \leq \frac{2\pi}{n},
 \]
 where \(\partial_\pm V = M \times \{\pm 1\}\).
\end{Conjecture}

Rosenberg and Stolz previously proposed a seemingly related conjecture, see~\cite[Section~7]{Rosenberg-Stolz:Manifolds-of-psc}:

\begin{Conjecture}\label{conj:no-complete-psc}
 Let \(M\) be a closed manifold of dimension \(\neq 4\) such that \(M\) does not admit a~metric of positive scalar curvature.
 Then \(M \times \R\) does not admit a complete metric of positive scalar curvature.
\end{Conjecture}

While these two conjectures appear similar, there is no direct formal implication between them.
Conjecture~\ref{conj:band-width} only implies that \(M \times \R\) does not admit a complete metric of \emph{uniformly} positive scalar curvature, a weaker conclusion than what is desired by Conjecture~\ref{conj:no-complete-psc}.
\begin{Remark}
We have to exclude dimension \(4\) because even this weaker conclusion is known to fail in this case: Using Seiberg--Witten invariants, it is possible to show that there exists a simply connected spin \(4\)-manifold such that \(M\) does not admit positive scalar curvature but \(M \times \Sphere^1\) does.
To this end, one can take a simply connected \(4\)-manifold of vanishing \(\Ahat\)-genus which does not admit a metric of positive scalar curvature, \cite[Counterexample~4.16]{RosenbergStolz:PSCSurgeryConnections}, \cite{Taubes:SeibergWitten}.
But if \(M \times \Sphere^1\) admits positive scalar curvature, then \(M \times \R\) admits a complete metric of uniformly positive scalar curvature.
\end{Remark}

The ideas behind these conjectures have their roots in work of Gromov and Lawson~\mbox{\cite{gromov-lawson:spin-and-scalar-in-presence, GromovLawson:PSCDiracComplete}}, where they studied various notions of \emph{largeness} for Riemannian manifolds, most notably \emph{enlargeability}, which serve as obstructions to the existence of positive scalar curvature metrics.
While these conditions are of purely geometric topological nature, the original proof that these are obstructions to positive scalar curvature rely on the spin condition and (relative) index theory of spinor Dirac operators.\footnote{An approach relying on the minimal hypersurface technique has been given by Cecchini and Schick~\cite{CecchiniSchick:Enlargeable}.}
They already showed that enlargeable spin manifolds satisfy Conjecture~\ref{conj:no-complete-psc}, see~\cite[Section~6]{GromovLawson:PSCDiracComplete}.
Gromov recently generalized enlargeability based on results related to Conjecture~\ref{conj:band-width} to a new notion of \emph{iso-enlargeability}~\cite[Section~4]{Gromov:MetricInequalitiesScalar} which deviates from classical enlargeability by using not necessarily complete manifolds (see Section~\ref{sec:largeness}).
Initially, it appeared that Conjecture~\ref{conj:band-width} and iso-enlargeability \emph{cannot} be understood using index theory of the Dirac operator because the latter does not naturally lend itself to the study non-complete manifolds.
Instead, in \cite{Gromov:MetricInequalitiesScalar} these ideas were approached via the minimal hypersurface technique of Schoen and Yau~\cite{SchoenYau:HypersurfaceMethod}.
On the other hand, as mentioned above, Conjecture~\ref{conj:no-complete-psc} was already known to be accessible to index-theoretic methods.
Cecchini~\cite{Cecchini:CalliasTypePSC} strengthened this link by showing that Conjecture~\ref{conj:no-complete-psc} holds whenever \(M\) is spin and has non-vanishing \emph{Rosenberg index} (see Section~\ref{sec:Rosenberg-index} for a survey of the Rosenberg index).
Furthermore, it turned out that Conjecture~\ref{conj:band-width} can also be addressed using index theory of Dirac operators.
Indeed, it was shown by the author~\cite{Zeidler:band-width-estimates} that a~version of Conjecture~\ref{conj:band-width} with a weaker upper bound holds if \(M\) has non-vanishing Rosenberg index.
Recently, Cecchini~\cite{Cecchini:LongNeck}, along with addressing several other questions in Gromov's metric inequalities programme, improved this estimate by showing that Conjecture~\ref{conj:band-width} holds as stated if \(M\) has non-vanishing Rosenberg index.

The purpose of this note is threefold.
\emph{Firstly}, we revisit the direct index-theoretic approaches to Conjectures~\ref{conj:band-width} and~\ref{conj:no-complete-psc} via the Rosenberg index.
This was prompted by observing that both cases are based on the same index-theoretic technique, namely an index theorem for certain \emph{Callias-type} operators, suggesting that there should be a more general geometric statement underlying both.
This is achieved in Theorem~\ref{thm:quantitative-codim1-obstr} below.
We show that in the presence of a~hypersurface \(M \subset X\) with an index-theoretic obstruction, positive scalar curvature \(\geq \sigma > 0\) in a~neighborhood \(\{p \in X \mid d(p, M) \leq d/2 \}\) of width \(d\) forces the scalar curvature somewhere else to be negative in a~quantifiable way depending on \(d\) and \(\sigma\).
We formulate this for the normalized case \(\sigma = n(n-1) = \scal_{\Sphere^n}\).

\begin{Theorem} \label{thm:quantitative-codim1-obstr}
 Let \(X\) be a spin \(n\)-manifold without boundary and \(M \subseteq X\) a closed codimension one submanifold with trivial normal bundle such that the induced map \(\pi_1 M \to \pi_1 X\) is injective.
 Suppose that the Rosenberg index \(\alpha(M) \in \KO_{n-1}(\Cstar \pi_1 M)\) does not vanish.

 Let \(g\) be a complete Riemannian metric on \(X\) which has scalar curvature bounded below by \(n(n-1)\) in a neighborhood of \(M\) of width \(d < 2\pi/n\).
 Then
 \begin{equation} \label{eq:negativity}
 \inf_{p \in X} \scal_g(p) \leq - n(n-1) \tan\left(\frac{nd}{4}\right)^2.
 \end{equation}
\end{Theorem}

For manifolds of non-vanishing Rosenberg index, this immediately implies Conjecture~\ref{conj:no-complete-psc} and furthermore, since the term on the right-hand side of \eqref{eq:negativity} tends to \(-\infty\) as \(d \to 2 \pi /n\), it also implies Conjecture~\ref{conj:band-width}.

\begin{Corollary}\label{cor:noPSc-and-bandWidth}
 Let \(M\) be a closed spin manifold of non-vanishing Rosenberg index \(\alpha(M) \neq 0 \in \KO_{n-1}(\Cstar \pi_1 M)\).
 Then the following holds:
 \begin{itemize}\itemsep=0pt
 \item \(M \times \R\) does not admit a complete metric of positive scalar curvature,
 \item Every Riemannian metric \(g\) on \(V = M \times [-1,1]\) of scalar curvature \(\geq n(n-1)\) satisfies
 \[
 \width(V, g) < \frac{2\pi}{n}.
 \]
 \end{itemize}
\end{Corollary}
Note that the latter is even slightly stronger than the desired conclusion of Conjecture~\ref{conj:band-width} because we obtain a strict inequality.

\emph{Secondly}, we exhibit an index-theoretic point of view towards Gromov's \emph{iso-enlargeability} for spin manifolds.
This was essentially already implicit in \cite{Zeidler:band-width-estimates}, where the related notion of \emph{infinite \(\mathcal{KO}\)-width} was introduced (see Section~\ref{sec:KO-width}).
Here we make this explicit and study a~slight generalization of iso-enlargeability, namely \emph{\(\Ahat\)-iso-enlargeability} (following Gromov and Lawson's \(\Ahat\)-enlargeability, see Section~\ref{sec:largeness}), and relate it to infinite \(\mathcal{KO}\)-width.
This leads to the following theorem (see Section~\ref{A-hat-largeness}).
\begin{Theorem}\label{thm:A-hat-iso-enlargeable}
 Let \(M\) be a closed spin manifold which is \(\Ahat\)-iso-enlargeable.
 Then \(M\) does not admit a metric of positive scalar curvature.
\end{Theorem}

\emph{Finally}, we survey several open questions suggested by these recent developments connecting index theory to metric and largeness aspects of scalar curvature (see Sections~\ref{sec:largeness} and~\ref{sec:Problems}).

\section{The Rosenberg index} \label{sec:Rosenberg-index}
In this section, we provide a brief survey of the aspects of the Rosenberg index relevant to our statements and open problems.
More material on these matters can be found in the original articles of Rosenberg~\cite{Rosenberg:PSCNovikovI,Rosenberg:PSCNovikovII,Rosenberg:PSCNovikovIII} and in \cite{RosenbergStolz:PSCSurgeryConnections,Schick:ICM,Stolz:ManifoldsOfPSC,Stolz98Concordance}.

The starting point is that the spinor Dirac operator on a spin manifold is closely related to the scalar curvature via the Schr\"odinger--Lichnerowicz formula \(\Dirac^2 = \nabla^\ast \nabla + \frac{\scal}{4}\).
Together with the Atiyah--Singer index theorem, this was first used by Lichnerowicz~\cite{Lichnerwociz:Spineurs} to show that the \(\Ahat\)-genus is an obstruction to the existence of a positive scalar curvature metric on a closed spin manifold \(M\).
Later, this connection between the Dirac operator and scalar curvature was exploited further using variations of the index invariant.
For closed spin manifolds, the most sophisticated such invariant is the \emph{Rosenberg index} \(\alpha(M) \in \KO_n(\Cstar \pi_1 M)\) which takes values in the (real) K-theory of the group \textCstar-algebra \(\Cstar \pi_1 M\).
Notable examples of manifolds with non-vanishing Rosenberg index (but which in general do not have non-vanishing \(\Ahat\)-genus) include (area-)enlargeable spin manifolds~\cite{Hanke-Kotschick-Roe-Schick:Coarse-Top-Enlargeability, HankeSchick:Enlargeable,HankeSchick:EnlargeableInfinite}.

We will now describe two (equivalent) points of view towards the Rosenberg index.
The first is as the index of the spinor Dirac operator of \(M\) twisted by the ``universal flat bundle'' on~\(M\).
The latter is the \emph{Mi\v{s}\v{c}enko bundle} \(\mathcal{L}_M \to M\) which is the flat bundle of finitely generated projective Hilbert-\(\Cstar \Gamma\)-modules associated to the representation of \(\Gamma \coloneqq \pi_1 M\) on \(\Cstar \Gamma\) by left-multiplication.
Index theory extends to the situation of elliptic operators with coefficients in \textCstar-algebras yielding index invariants in corresponding operator K-theory groups.
This goes back to Mi\v{s}\v{c}enko and Fomenko~\cite{Mishchenko-Fomenko:Index}; for more recent expositions of the technical background relevant to our situation we refer to~\cite{Ebert:EllipticRegularityDirac,HankePapeSchick:CodimensionTwoIndex}.
The Rosenberg index \(\alpha(M) \in \KO_n(\Cstar \pi_1 M)\) is then simply the (K-theoretic) index of the twisted Dirac operator on \(\SpinBdl_M \otimes \mathcal{L}_M\) using the flat connection induced on the \Mishchenko\ bundle \(\mathcal{L}_M\), and \(\SpinBdl_M\) is the spinor bundle.
The Schr\"odinger--Lichnerowicz formula extends to this situation so that the non-vanishing of \(\alpha(M) \in \KO_n(\Cstar \Gamma)\) is still an obstruction to the existence of a positive scalar curvature metric on \(M\).

The second point of view is through the lens of the strong Novikov and Baum--Connes conjectures.
The (real) Baum--Connes conjecture features the assembly map
\[
 \mu \colon \ \KO_\ast^\Gamma(\Eub \Gamma) \to \KO_\ast(\Cstar \Gamma),
\]
where the left-hand side is the equivariant KO-homology of the classifying space for proper actions.
Note that there is a canonical map \(\iota_\ast \colon \KO_\ast(\Bfree \Gamma) = \KO_\ast^\Gamma(\Efree \Gamma) \to \KO_\ast^\Gamma(\Eub \Gamma)\), where \(\Bfree \Gamma\) is the usual classifying space of \(\Gamma\) and \(\Efree \Gamma\) the classifying space for free actions.
The map \(\iota_\ast\) is always rationally injective.
Moreover, it is an isomorphism iff \(\Gamma\) is torsion-free.
The \emph{Novikov assembly map} is the composition
\[
 \nu = \mu \circ \iota_\ast \colon \ \KO_\ast^\Gamma(\Bfree \Gamma) \to \KO_\ast(\Cstar \Gamma).
\]
The \emph{strong Novikov conjecture} asserts that \(\nu\) is rationally injective.
If \(M\) is a spin manifold, then we have the KO-fundamental class \([M]_{\KO} \in \KO_n(M)\).
Suppose that we have a map \(f \colon M \to \Bfree \Gamma\), then \(\nu(f_\ast[M]_{\KO}) \in \KO_n(\Cstar \Gamma)\) equals the index of the Dirac operator of \(M\) twisted by \(f^\ast \mathcal{L}_\Gamma\), where \(\mathcal{L}_\Gamma \to \Bfree \Gamma\) is the \Mishchenko-bundle.
In particular, the Rosenberg index is given as
\[
 \alpha(M) = \nu(c_\ast [M]_{\KO}) \in \KO_n(\Cstar \pi_1 M),
\]
where \(c \colon M \to \Bfree \pi_1 M\) is the classifying map of the fundamental group.
Together with the strong Novikov conjecture, this reduces the study of the Rosenberg index to computations in (KO-)homology.
If we are only interested in rational information, we can apply the \emph{Pontryagin character} to reduce this further to ordinary (co)homology.
The (homological) Pontryagin character is the composition of the (homological) Chern character with complexification and yields an isomorphism
\[
 \ph \colon \ \KO_n(X) \otimes \Q \xrightarrow{\cong} \bigoplus_{k \in \Z} \HZ_{n + 4k}(X; \Q).
\]
Applied to the KO-fundamental class this yields \(\ph([M]_{\KO}) = [M] \cap \AhatClass(M)\),
where \([M] \in \HZ_n(M)\) is the ordinary fundamental class and \(\AhatClass(M) \in \HZ^{4\ast}(M; \Q)\) the total \(\Ahat\)-class of \(M\).
Given any map \(f \colon M \to \Bfree \Gamma\) and an ordinary cohomology class \(\theta \in \HZ^{k}(\Bfree \Gamma; \Q)\), we have
\begin{equation} \label{eq:higherAhat}
 \langle \ph(f_\ast [M]_{\KO}), \theta \rangle = \big\langle [M] \cap \AhatClass(M), f^\ast \theta \big\rangle = \big\langle [M], \AhatClass(M) \cup f^\ast \theta \big\rangle,
\end{equation}
where a number as on the right-hand side is called a \emph{higher \(\Ahat\)-genus} of \(M\).
The strong Novikov conjecture thus implies that the non-vanishing of any higher \(\Ahat\)-genus is an obstruction to positive scalar curvature.

In the case of a trivial fundamental group, the Rosenberg index reduces to Hitchin's \(\alpha\)-invariant~\cite{Hitchin:HarmonicSpinors}.
Stolz~\cite{Stolz:SimplyConnected} proved that a simply connected spin manifold of dimension \(\geq 5\) admits a positive scalar curvature metric if and only if its \(\alpha\)-invariant vanishes.
Note that dimensions \(1,2 \mod 8\) this is a \(\Z/2\)-valued invariant, so here it is important to consider the full integral information.
More generally, one might expect that the analogous statement holds for not necessarily simply connected manifolds using the (integral) Rosenberg index.
This is the content of the (unstable) Gromov--Lawson--Rosenberg conjecture.
Unfortunately, this conjecture is false in general~\cite{Schick:Counterexample}.
However, in its stead, one can consider a \emph{stable} version of the Gromov--Lawson--Rosenberg conjecture as introduced by Rosenberg and Stolz~\cite{RosenbergStolz:StableGLR}.
The stable conjecture is motivated by the observation that the receptacle of the Rosenberg index is \(8\)-periodic due to Bott periodicity, whereas the geometric problem concerning the existence of positive scalar curvature metrics does not feature such a periodicity in an obvious way.
To remedy this mismatch, one stabilizes the positive scalar curvature existence problem:
To this end, let \(B\) be a \emph{Bott manifold}, that is, a simply connected spin \(8\)-manifold such that \(\Ahat(B) = 1\).
Note that \(B\) can be chosen to be endowed with a Ricci flat metric~\cite{Joyce:Compact8Mfd}.
Moreover, using Bott periodicity, \(\alpha(M \times B^k) = \alpha(M) \in \KO_{n}(\Cstar \Gamma)\) for any \(k \in \N\).
We say that a manifold \emph{stably admits a positive scalar curvature metric} if there exists some \(k \in \N\) such that \(M \times B^k\) admits a positive scalar curvature metric.
\begin{Conjecture}[stable GLR conjecture] \label{conj:stable-glr}
 A spin manifold stably admits a positive scalar curvature metric if and only if \(\alpha(M) = 0 \in \KO_n(\Cstar \Gamma)\).
\end{Conjecture}

Stolz proved that the stable conjecture holds whenever the (real version of the) Baum--Connes assembly map is injective:

\begin{Theorem}[Stolz~{\cite[Theorem~3.10]{Stolz:ManifoldsOfPSC}}] \label{thm:Stolz-Stable-GLR}
 The stable GLR conjecture holds if the Baum--Connes assembly map \(\mu \colon \KO_\ast^\Gamma(\Eub \Gamma) \to \KO_\ast(\Cstar \Gamma)\) is injective.
\end{Theorem}

This implies that there is a large scope of cases, where the stable GLR conjecture is known to hold, and that any counterexample would also be a counterexample to the Baum--Connes conjecture.
In a similar vein, Schick proposed a meta-conjecture~\cite[Conjecture~1.5]{Schick:ICM} asserting that every positive scalar curvature obstruction ``based on index theory of Dirac operators'' can be obtained from the Rosenberg index.

We conclude this section with the discussion of two subtle technicalities we have ignored so far.
\begin{Remark}[maximal versus reduced]
 The group \textCstar-algebras \(\Cstar\Gamma\) comes in (at least) two flavours, depending on which completion of the group ring is used.
 We have the \emph{reduced} group \textCstar-algebra \(\CstarRed \Gamma\) (which is the completion induced by the left-regular representation of \(\Gamma\)) and the \emph{maximal} group \textCstar-algebra \(\CstarMax \Gamma\) (which is the universal representation of \(\Gamma\)).
 In the discussion above, we have not stressed which one we use because for our purposes we can usually work with either one of them.
 First of all, there is a canonical quotient map \(\CstarMax \Gamma \to \CstarRed \Gamma\) which means that the maximal \textCstar-algebra a priori could potentially capture more obstructions.
 Moreover, viewing the \Mishchenko\ bundle as the ``universal flat bundle'', the maximal \textCstar-algebra is the more natural choice because any unitary representation of \(\Gamma\) factors through \(\CstarMax \Gamma\).
 On the other hand, the Baum--Connes conjecture is usually formulated and studied for the reduced group \textCstar-algebra (because surjectivity definitely fails in the maximal case).
 Fortunately, injectivity of the Baum--Connes assembly map for the reduced case implies injectivity in the maximal case and so the distinction becomes less relevant for our purposes even when we view the Rosenberg index through this lens.
\end{Remark}

\begin{Remark}[real versus complex]
 In the literature on the Baum--Connes conjecture, usually the case of complex K-theory and -homology is treated, whereas for the purposes of spin geometry and scalar curvature, real K-theory and -homology is the more appropriate choice.
 Fortunately, whether or not the Baum--Connes assembly map is an isomorphism does not depend on this choice,
 and after inverting \(2\), this is even true for injectivity and surjectivity separately~\cite{Baum-Karoubi:real-BC,Schick:real-vs-complex}.
 In particular, for the strong Novikov conjecture as formulated here the difference does not matter.
\end{Remark}

\section{Callias-type operators and codimension one obstructions} \label{sec:codim-1-obstr}
In this section, we will provide a proof of Theorem~\ref{thm:quantitative-codim1-obstr}.
This is a simplified synthesis of the previous approaches from \cite{Cecchini:CalliasTypePSC}, \cite[Section~6]{Cecchini:LongNeck} and \cite[Section~2]{Zeidler:band-width-estimates}.
We rely on the same index-theoretic setup which has its roots in work of Higson~\cite{Higson:CobordismInvariance} and Bunke~\cite{Bunke:RelativeIndexCallias}.

In this section, we fix an arbitrary unital Real \textCstar-algebra \(A\).
For the purposes of the main theorem of this note, one can keep the case \(A = \Cstar \Gamma\) with \(\Gamma\) some discrete group in mind.

Let \(W\) be a complete spin \(n\)-manifold together with a proper smooth \(1\)-Lipschitz function \(x \colon W \to \R\).
Fix some regular value \(a \in \R\) of \(x\) and set \(M \coloneqq x^{-1}(a)\).
Suppose furthermore that \(W\) is endowed with a smooth bundle \(E \to W\) of finitely generated projective Hilbert-\(A\)-modules furnished with a metric connection.
Let \(f \colon \R \to \R\) be some non-decreasing smooth function satisfying \(f(0) = 0\) and such that \(f^2 - f'\) is bounded below by a positive constant outside a~compact subset of~\(\R\).
In particular, the latter property includes the following two special cases:
\begin{enumerate}\itemsep=0pt
 \item[(i)] the function \(f\) is proper and Lipschitz,
 \item[(ii)] or there exists \(R > 0\) such that \(f(x) = \lambda_-\) for all \(x \leq -R\) and \(f(x) = \lambda_+\) for all \(x \geq R\), where \(\lambda_- < 0 < \lambda_+\) are constants.
\end{enumerate}
Denote the \(\Cl_n\)-linear spinor bundle of \(W\) by \(\SpinBdl_W\) and let \(\Dirac_{W,E}\) be the Dirac operator twisted by \(E\).
We then consider the operator
\[
 B = \Dirac_{W, E} \tensgr 1 + f(x) \tensgr \epsilon
\]
acting on compactly supported smooth sections of \(\SpinBdl_W \tensgr E \tensgr \Cl_{0,1}\).
Then \(B\) is an essentially self-adjoint regular Fredholm operator.
It has an index \(\ind(B) \in \KO_{n-1}(A)\) which satisfies
\[
 \ind(B) = \ind(\Dirac_{M, E|_M}) \in \KO_{n-1}(A).
\]
Moreover, we have the formula
\begin{equation}
B^2 = \Dirac^2_{W,E} \tensgr 1 + f'(x) \clm(\D x) \tensgr \epsilon + f(x)^2 \geq \Dirac^2_{W,E} \tensgr 1 + f(x)^2 - f'(x),
\label{eq:callias-estimate}
\end{equation}
where \(\clm\) denotes the Clifford action and we have used \(\|\mathrm{d} x\| \leq 1\) and \(f' \geq 0\).
For all of this we refer, e.g., to \cite[Section~6]{Cecchini:LongNeck}, \cite[Appendix~A]{Zeidler:band-width-estimates}.

\begin{Theorem}\label{thm:quantitative-codim-1}
 In the above setup, suppose that \(\ind(\Dirac_{M, E|_M}) \neq 0 \in \KO_{n-1}(A)\).
 Let \(I \subseteq \R\) be an interval of length \(d < 2 \pi / n\) such that the scalar curvature of \(W\) is bounded below by \(n(n-1)\) on the subset \(x^{-1}(I)\).
 Then
 \begin{equation*}
 \inf_{p \in W} \scal(p) \leq - n(n-1) \tan\left(\frac{nd}{4}\right)^2.
 \end{equation*}
 \end{Theorem}
 \begin{proof}
 We can assume without loss of generality that \(I = [-d/2,d/2]\).
 Then suppose, by contraposition, that there exists \(\sigma < \tan\big(\frac{nd}{4}\big)^2\) such that \(\scal(p) \geq -n(n-1) \sigma\) for all \(p \in W\).
 Choose \(0 < \tilde{d} < d\) such that \(\sigma < \tan\big(\frac{n\tilde{d}}{4}\big)^2 < \tan\big(\frac{nd}{4}\big)^2\) and \(0 < r < n/2\) such that
 \begin{equation}
 \frac{n^2 \sigma}{4} < r^2 \tan\big(r\tilde{d}/2\big)^2. \label{eq:choose-r}
 \end{equation}
 Choose a smooth function \(\xi \colon \R^2 \to \R\) which satisfies \(0 \leq \xi(x,y) \leq y^2 + r^2\) for all \((x,y) \in \R^2\), \(\xi(x,y) = y^2 + r^2\) for \(|x| \leq \tilde{d}/2\) and \(\xi(x,y) = 0\) for \(|x| \geq d/2\).
 Let \(f\) be the unique maximal solution of the initial value problem
 \begin{gather*}
 f'(x) = \xi(x,f(x)), \\
 f(0) = 0.
 \end{gather*}
 By construction, \(f\) is a smooth non-decreasing function defined on some interval containing \(0\).
 We first claim that \(f\) is defined on the entire real line.
 To this end, consider the auxilliary function \(\bar{f} \colon (-\pi/(2 r), \pi/(2r)) \to \R\), \(\bar{f}(x) = r \tan(r x)\).
 Note that this is the unique solution of \(\bar{f}'(x)= \bar{f}(x)^2 + r^2\), \(\bar{f}(0) = 0\) and that \(\pi/(2r) > \pi /n > d/2\).
 By the choice of \(\xi\), \(f'(x) \leq f(x)^2 + r^2\) which implies that \(|f| \leq |\bar{f}|\) as long as both are defined.
 This implies that \(f\) exists at least for all \(|x| \leq d/2\).
 But \(\xi(x,y) = 0\) for \(|x| \geq d/2\), so \(f\) must continue for all times and remain constant on each component of \(\R \setminus (-d/2,d/2)\).
 Moreover, by monotonicity and because \(f(x) = \bar{f}(x)\) for \(|x| \leq \tilde{d}/2\), we obtain the estimate
 \begin{equation}
 |f(x)| \geq |f(\pm \tilde{d}/2)| = \bar{f}(\tilde{d}/2) = r \tan\big(r \tilde{d} / 2\big) \qquad \text{for all \(|x| \geq d/2\).}
 \label{eq:outside-bound}
 \end{equation}

 Now we consider the operator
 \(
 B = \Dirac \tensgr 1 + f(x) \tensgr \epsilon
 \)
 as above, where \(\Dirac = \Dirac_{W,E}\).
 Then \linebreak using~\eqref{eq:callias-estimate} together with the Schr\"odinger--Lichnerowicz formula and the Friedrich estimate (compare~\cite[Section~3.2]{Cecchini:LongNeck}), we obtain
 \begin{align*}
 B^2 &\geq \Dirac^2 \otimes 1 + f(x)^2 - f'(x) \\
 &\geq {\frac{n \scal}{4(n-1)} + f(x)^2 - f'(x)}\ \eqqcolon \Psi.
 \end{align*}
 We denote the latter function by \(\Psi \colon W \to \R\) and will show that it is bounded below by a positive constant to finish the proof.
 By construction of \(f\), we always have \(f(x)^2 - f'(x) \geq -r^2\).
 For \(|x| \leq d/2\), this implies together with the scalar curvature bound on \(x^{-1}(I)\) that
 \[
 \Psi(p) \geq \frac{n^2}{4} - r^2 > 0 \qquad \text{if \(|x(p)| \leq d/2\).}
 \]
 Moreover, for \(|x| \geq d/2\), we have \(f'(x) = 0\) and so \eqref{eq:choose-r} and \eqref{eq:outside-bound} imply
 \begin{align*}
\Psi(p) \geq - \frac{n^2 \sigma}{4} + f(x(p))^2 \geq - \frac{n^2 \sigma}{4} + r^2 \tan\big(r \tilde{d} / 2\big)^2 > 0 \qquad \text{if \(|x(p)| \geq d/2\).}
 \end{align*}
 This proves that \(B^2 \geq \inf\limits_{p \in W} \Psi(p) > 0\) and hence \(\ind(\Dirac_{M, E|_M}) = \ind(B) = 0\).
 \end{proof}
 \begin{Remark}
 The argument implies the following more general observation:
 If \(\ind(\Dirac_{M, E|_M})\) does not vanish, then it follows that
 \[
 \inf_{p \in W} \left( \frac{n \scal(p)}{4(n-1)} + f(x(p))^2 - f'(x(p)) \right) \leq 0
 \]
 for any such smooth function \(f \colon \R \to \R\) allowed in this setup.
 By choosing different functions one can thereby obtain various pointwise constraints on how scalar curvature can be distributed along \(W\).
 \end{Remark}

 Note that this implies Theorem~\ref{thm:quantitative-codim1-obstr} from the introduction by taking \(W\) to be the connected covering of \(X\) such that \(\pi_1 W = \pi_1 M\), the bundle \(E = \mathcal{L}_W\) to be the \Mishchenko-bundle, and the function \(x\) to be (a smooth approximation of) the signed distance function to \(M \subseteq W\).

Rescaling the scalar curvature lower bound on the tubular neighborhood implies the first part of Corollary~\ref{cor:noPSc-and-bandWidth}.
 How the second part of Corollary~\ref{cor:noPSc-and-bandWidth} follows from this theorem will be explained in Corollary~\ref{cor:KOBandEstimate} below.

\section[Ahat-largeness]{\(\boldsymbol{\Ahat}\)-largeness} \label{sec:largeness}
In this section, we review various notions of largeness for Riemannian manifolds following Gromov and Lawson~\cite{Gromov:MetricInequalitiesScalar, gromov-lawson:spin-and-scalar-in-presence,GromovLawson:PSCDiracComplete}.
For a textbook treatment see also~\cite[Chapter~IV, Section~5]{LawsonMichelsohn:SpinGeometry}.

We need the notion of the (\(\Ahat\)-)mapping degree in different setups.
To this end, let \(f \colon X \to Y\) be a smooth map between oriented manifolds (possibly with boundary) and assume that we are in one of the two following situations:
\begin{enumerate}\itemsep=0pt
 \item[(i)] the map \(f \colon X \to Y\) is proper and takes \(\partial X\) to \(\partial Y\),
 \item[(ii)] or \(Y\) is a closed manifold with a fixed base-point \(\ast \in Y\) and \(f\) is \emph{constant at infinity}, that is, it maps both \(\partial X\) and the complement of a compact subset to the base-point \(\ast \in Y\). 
\end{enumerate}
In these situations and if \(\dim X = \dim Y = n\), then \(f\) has a well-defined mapping degree \(\deg(f) \in \Z\).
It can be characterized in terms of differential forms, that is,
\[
 \int_X f^\ast \omega = \deg(f)\ \int_{Y} \omega
\]
for all compactly supported \(n\)-forms \(\omega\) on \(Y\) which vanish on \(\partial Y\).
More generally, if \(\dim X = n\) and \(\dim Y = k\) are not necessarily equal, one defines the \(\Ahat\)-degree of \(f\) by the formula
\[
 \int_X \AhatClass(X) \wedge f^\ast \omega = \AhatDeg(f)\ \int_Y \omega
\]
for all compactly supported \(k\)-forms \(\omega\) on \(Y\) which vanish on \(\partial Y\).
Here \(\AhatClass(X)\) denotes the total \(\Ahat\)-class of \(X\).
Alternatively, the \(\AhatDeg(f)\) can be described as the \(\Ahat\)-genus of the transversal pre-image \(f^{-1}(y)\) of a point \(y \in Y\) (which in case (ii) must not be equal to the base-point).
Note that the \(\Ahat\)-degree can only be non-zero if \(k \leq n\).

In the following, we fix a base-point \(\ast \in \Sphere^k\) in the sphere of each dimension \(k \geq 0\).

\begin{Definition}
 Let \(\varepsilon > 0\).
 We say a connected oriented Riemannian manifold \(X\) (possibly with boundary) of dimension \(n\) is called \emph{\(\varepsilon\)-\(\Ahat\)-hyperspherical} in dimension \(k\) if there exists an \(\varepsilon\)-Lipschitz map \(f \colon X \to \Sphere^k\) which is constant at infinity and has non-zero \(\Ahat\)-degree.
\end{Definition}

\begin{Remark}
 If \(k = 0\) in the definition above, then \(X\) must be a closed manifold such that \(\Ahat(X) \neq 0\).
\end{Remark}

We now review various notions of largeness.
The first three go back in one form or another to Gromov and Lawson~\cite{gromov-lawson:spin-and-scalar-in-presence,GromovLawson:PSCDiracComplete}, whereas \emph{iso-enlargeability} was introduced more recently by Gromov~\cite{Gromov:MetricInequalitiesScalar}.
The notion of \emph{\(\Ahat\)-iso-enlargeability} defined here is a straightforward generalization of the latter.

\begin{Definition}[largeness properties]
 Let \(X\) be a connected oriented Riemannian manifold of dimension \(n\) without boundary. We say that \(X\) is
 \begin{itemize}\itemsep=0pt
 \item \(\Ahat\)-\emph{hypereuclidean} in dimension \(k\), if there exists a proper Lipschitz map \(X \to \R^k\) of non-zero \(\Ahat\)-degree,
 \item \(\Ahat\)-\emph{hyperspherical} in dimension \(k\), if \(X\) is \(\varepsilon\)-\(\Ahat\)-hyperspherical in dimension \(k\) for each \(\varepsilon > 0\),
 \item \(\Ahat\)-\emph{enlargeable} in dimension \(k\), if for each \(\varepsilon >0 \), there exists a connected Riemannian covering \(\bar{X} \to X\) such that \(\bar{X}\) is \(\varepsilon\)-\(\Ahat\)-hyperspherical in dimension \(k\),
 \item \(\Ahat\)-\emph{iso-enlargeable} in dimension \(k\), if for each \(\varepsilon > 0 \), there exists a Riemannian local isometry \(\bar{X} \to X\) (where \(\bar{X}\) is allowed to be incomplete or have a boundary) such that \(\bar{X}\) is \(\varepsilon\)-\(\Ahat\)-hyperspherical in dimension \(k\).
 \end{itemize}
\end{Definition}
If \(k = 0\) in any of these notions, then this just means that \(X\) is closed and has non-zero \(\Ahat\)-genus.
In the other extreme case \(k = n\), the \(\Ahat\)-degree reduces the usual degree and this yields the usual notions of hypereuclidean, hyperspherical, enlargeable and iso-enlargeable.
If we omit the dimension \(k\) when using any of these notions, we will mean that the corresponding property holds in some dimension \(0 \leq k \leq n\).

The following implications hold.
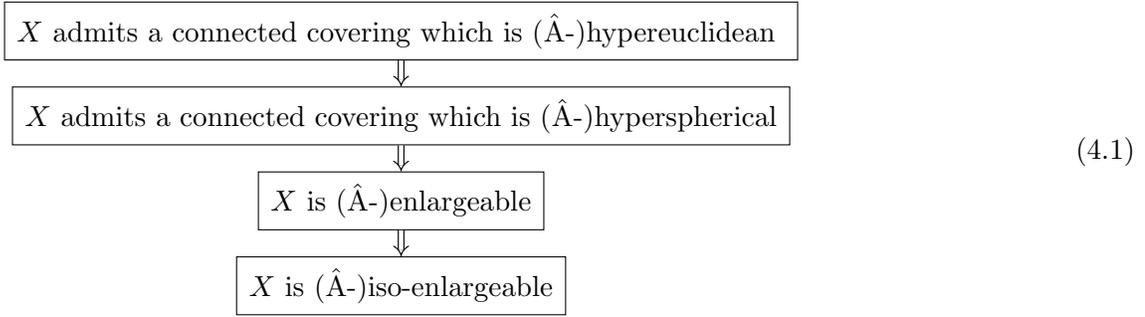
\begin{equation}\label{eq:largeness-implications}
\begin{tikzcd}[cells={nodes={draw=black, rectangle,anchor=center,minimum height=2em}}, row sep=small]
 \text{\(X\) admits a connected covering which is (\(\Ahat\)-)hypereuclidean } \dar[Rightarrow]\\
 \text{\(X\) admits a connected covering which is (\(\Ahat\)-)hyperspherical} \dar[Rightarrow]\\
 \text{\(X\) is (\(\Ahat\)-)enlargeable}\dar[Rightarrow] \\
 \text{\(X\) is (\(\Ahat\)-)iso-enlargeable}
\end{tikzcd}
\end{equation}
The first is due to the fact that Euclidean space \(\R^n\) is hyperspherical, and the others follow immediately from the definitions.

However, it is an open question whether any of these implications admit a converse for closed manifolds.
This is already the case for the non-\(\Ahat\)-variants of these notions.
Indeed, the question of whether every closed iso-enlargeable manifold is enlargeable was already posed by Gromov in~\cite[p.~659]{Gromov:MetricInequalitiesScalar}.
Brunnbauer and Hanke~\cite[Theorem~1.4]{BrunnbauerHanke} showed that there exist enlargeable closed manifolds whose \emph{universal} covering is not hyperspherical.
However, their construction does not prove the non-existence of an intermediate hyperspherical (or even hypereuclidean) covering.
It appears that given what we know at present we cannot even exclude the possibility that every iso-enlargeable manifold might admit a hypereuclidean covering, even though this looks unlikely.

\section{Index-theoretic notions of width} \label{sec:KO-width}
In this section, we say that a \emph{band} is a compact manifold \(V\), the boundary of which is decomposed into distinguished parts \(\partial V = \partial_- V \sqcup \partial_+ V\), where \(\partial_\pm V\) are unions of components.
This notion appeared first in \cite{Gromov:MetricInequalitiesScalar}, where such objects are called ``compact proper bands''.
A \emph{band map} is a~smooth map \(f \colon V \to V^\prime\) between bands such that \(f(\partial_\pm V) \subseteq  \partial_\pm V'\).

\begin{Definition} Let \((X,g)\) be some Riemannian manifold and \(\mathcal{V}\) some class of bands.
 We set
 \[\width_{\mathcal{V}}(X, g) \in [0, \infty]\]
 to be the (possibly infinite) supremum of all \( \width(V, \phi^\ast g)\), where \((V, \phi)\) varies through all \(V \in \mathcal{V}\) and local diffeomorphisms \(\phi \colon V \to X\).
 If no such local diffeomorphism exists, we define \(\width_{\mathcal{V}}(X, g) = 0\).
\end{Definition}

We consider the following examples of classes of bands.
\begin{description}\itemsep=0pt
 \item[ \(\hat{\mathcal{T}}\):]the class of all \emph{overtorical} bands~\cite{Gromov:MetricInequalitiesScalar}, that is, each \(V \in \hat{\mathcal{T}}\) admits a smooth band map \(V \to \Torus^{n-1} \times [-1,1]\) of non-zero degree.
 \item[ \(\hat{\mathcal{T}}_{\Ahat}\):]the class of all \emph{\(\Ahat\)-overtorical} bands, that is, each \(V \in \hat{\mathcal{T}}_{{\Ahat}}\) admits a smooth band map \(V \to \Torus^{k-1} \times [-1,1]\) of non-zero \(\Ahat\)-degree.
 \item[\(\mathcal{KO}\):] the class of all \(\mathcal{KO}\)-bands~\cite{Zeidler:band-width-estimates}, that is, each \(V \in \mathcal{KO}\) is spin and admits a flat bundle \(E \to V\) of finitely generated projective Hilbert-\(A\)-modules for some unital Real \textCstar-algebra \(A\) such that the twisted Dirac operator on \(\partial_\pm V\) has non-vanishing \(\ind(\Dirac_{\partial_- V, E|_{\partial_- V}}) \neq 0 \in \KO_{n-1}(A)\).
\end{description}

\begin{Proposition} \label{prop:overtorical-implies-KO}
 A band \(V \in \hat{\mathcal{T}}_{\Ahat}\) that is spin is a \(\mathcal{KO}\)-band.
\end{Proposition}
\begin{proof}
 Let \(V\) be an \(\Ahat\)-overtorical band.
 That means that there exists a band map \(f \colon V \to \Torus^{k-1} \times [-1,1]\) of non-zero \(\Ahat\)-degree.
 Let \(\mathcal{L}_{\Torus^{k-1}} \to \Torus^{k-1} \times [-1,1]\) be the Mi\v{s}\v{c}enko\ bundle of \(\Torus^{k-1} \times [-1,1]\).
 We set \(E \coloneqq f^\ast \mathcal{L}\).
 This is a flat bundle of finitely generated projective Hilbert-\(\Cstar(\Z^{k-1})\)-modules over \(V\).
 We will show that \(\ind(\Dirac_{\partial_- V, E})\) is rationally non-zero.
 To this end, observe that
 \[
 \ind(\Dirac_{\partial_- \bar{V}, E}) = \nu(f_\ast [\partial_- V]_{\KO}),
 \]
 where \(\nu\) is the Novikov assembly map of \(\Z^{k-1}\).
 Since \(\Z^{k-1}\) satisfies the strong Novikov conjecture, it is enough to show that
 \(f_\ast [\partial_- V]_{\KO}\) is rationally non-zero.
 Indeed, it follows from \eqref{eq:higherAhat} that
 \[
 \langle \ph(f_\ast [\partial_- V]_{\KO}), \theta \rangle = \AhatDeg(f|_{\partial_- V}) = \AhatDeg(f) \neq 0
 \]
 for each cohomology class \(\theta \in \HZ^{k-1}\big(\Torus^{k-1}\big)\) such that \(\big\langle \big[\Torus^{k-1}\big], \theta \big\rangle = 1\).
\end{proof}

\begin{Corollary}\label{cor:KO-width-bound-Ahat}
 A Riemannian spin manifold \((X,g)\) satisfies
 \[
 \width_{\mathcal{KO}}(X, g) \geq \width_{\hat{\mathcal{T}}_{\Ahat}}(X, g).
 \]
\end{Corollary}

Given a closed manifold \(M\), then whether or not \(\width_{\mathcal{V}}(M, g) = \infty\) does not depend on the choice of the Riemannian metric \(g\).
Thus, if \(\width_{\mathcal{V}}(M, g) = \infty\) for some Riemannian metric \(g\), we say that \(M\) has \emph{infinite \(\mathcal{V}\)-width}.
More generally:
\begin{Definition}
 Fix some class of bands \(\mathcal{V}\).
 We say that a manifold \(X\) has \emph{infinite \(\mathcal{V}\)-width} if \(\width_{\mathcal{V}}(X, g) = \infty\) for every complete Riemannian metric \(g\) on \(X\).
\end{Definition}

As has already been proved in \cite{Zeidler:band-width-estimates}, infinite \(\mathcal{KO}\)-width is an obstruction to positive scalar curvature.
We reprove this result with the optimal constant using Theorem~\ref{thm:quantitative-codim-1}.

\begin{Proposition} \label{cor:KOBandEstimate}
 Let \((V,g)\) be a Riemannian \(\mathcal{KO}\)-band such that the scalar curvature of \(g\) is bounded below by \(n(n-1)\).
 Then
 \[
 \width(V,g) < \frac{2 \pi}{n}.
 \]
 \end{Proposition}
 \begin{proof}
 Suppose, by contradiction, that \(\width(V,g) \geq \frac{2\pi}{n}\).
 We then attach infinite cylinders to each distinguished boundary part of \(V\) as in the proof of \cite[Theorem~3.3]{Zeidler:band-width-estimates}.
 This yields a~manifold
 \[W = (\partial_- V \times (-\infty, 0]) \cup_{\partial_- V} V \cup_{\partial_+ V} (\partial_+ V \times [0,\infty)).\]
 We endow \(W\) with a complete Riemannian metric of bounded geometry which agrees with~\(g\) on~\(V\).
 Let \(x \colon W \to \R\) be the signed distance function to \(\partial_- V\).

 Assume for simplicity that \(x\) is smooth.
 Then this fits into the setup of Theorem~\ref{thm:quantitative-codim-1} and satisfies \(\scal(p) \geq n(n-1)\) for all \(p \in x^{-1}(I)\) for an interval \(I\) of length \(\geq 2 \pi / n\).
 Now Theorem~\ref{thm:quantitative-codim-1} implies that
 \[
 \inf_{p \in W} \scal(p) \leq - n(n-1) \tan\left(\frac{nd}{4}\right)^2
 \]
 for every \(d < 2 \pi /n\).
 Consequently, \(\inf\limits_{p \in W} \scal(p) = -\infty\), a contradiction to the fact that the metric on \(W\) has bounded geometry.

 If \(x\) is not smooth, then the same argument can be applied to suitable smooth approximations of \(x\) (as is done in similar situations in \cite{Cecchini:LongNeck,Zeidler:band-width-estimates}) to obtain the same contradictory conclusion \(\inf\limits_{p \in W} \scal(p) = -\infty\).
 \end{proof}

\begin{Corollary}\label{cor:KO-width-obstr}
 Let \(X\) be a manifold of infinite \(\mathcal{KO}\)-width.
 Then \(X\) does not admit a complete metric of uniformly positive scalar curvature.
\end{Corollary}

The following families of examples have infinite \(\mathcal{KO}\)-width:
\begin{Proposition}\label{prop:boundary-obstr}
 Let \(W\) be a manifold with boundary such that \(\partial W\) is spin and \(\alpha(\partial W) \neq 0 \in \KO_\ast(\Cstar \pi_1 \partial W)\).
 Then the interior \(X = W^\circ\) has infinite \(\mathcal{KO}\)-width.
 In particular, \(X\) does not admit a complete metric of uniformly positive scalar curvature.
\end{Proposition}
\begin{proof}
 Let \(\phi \colon \partial W \times [0,2) \hookrightarrow W\) be a collar neighborhood such that \(\phi(\partial W \times \{0\}) = \partial W\).
 For each \(\varepsilon >0\), we have that \(\partial W \times [\varepsilon, 1]\) is a \(\mathcal{KO}\)-band, and
 \[
 \lim_{\varepsilon \to 0} \width(\partial W \times [\varepsilon, 1], \phi^\ast g) = \infty
 \]
 for any fixed complete metric \(g\) on \(X = W^\circ\).
 This proves that \(X\) has infinite \(\mathcal{KO}\)-width.
\end{proof}
The final conclusion of this theorem can also be obtained through \emph{coarse index theory} (see Section~\ref{subsec:coarse}) and was previously known.
Moreover, the arguments given in~\cite[Theorem~3.2]{ChangWeinbergerYu:Taming3} imply:
\begin{Corollary}
 In any dimension \(n \neq 3\), there exist contractible manifolds of infinite \(\mathcal{KO}\)-width.
\end{Corollary}

We mention two further examples.

\begin{Example}[{\cite[Example~4.8]{Zeidler:band-width-estimates}}] \label{ex:codim-1-obstr}
 Let \(X\) be a spin manifold and \(M \subset X\) be a codimension one submanifold with trivial normal bundle such that the induced map \(\pi_1 M \to \pi_1 X\) is injective.
 Suppose that \(\alpha(M) \neq 0 \in \KO_\ast(\Cstar \pi_1 M)\).
 Then \(X\) has infinite \(\mathcal{KO}\)-width.
\end{Example}

In the situation of Example~\ref{ex:codim-1-obstr}, applying Theorem~\ref{thm:quantitative-codim1-obstr} even shows the stronger conclusion that \(X\) does not admit any complete metric of (not necessarily uniform) positive scalar curvature.
Such a situation was also studied in \cite{NitscheSchickZeidler:Transfer,Zeidler:IndexObstructionPositive}, where it was shown that if \(X\) is closed, then it has non-vanishing Rosenberg index.

\begin{Example}[{\cite[Example~4.9]{Zeidler:band-width-estimates}}]\label{ex:codim-2-obstr}
 Let \(X\) be a spin manifold which admits a codimension two submanifold \(N \subseteq X\) with trivial normal bundle such that the induced maps \(\pi_1 N \hookrightarrow \pi_1 X\) and \(\pi_2 N \twoheadrightarrow \pi_2 X\) are injective and surjective, respectively.
 Suppose that \(\alpha(N) \neq 0 \in \KO_\ast(\Cstar \pi_1 N)\).
 Then \(X\) has infinite \(\mathcal{KO}\)-width.
\end{Example}

The situation of Example~\ref{ex:codim-2-obstr} was previously studied by Hanke, Pape, and Schick~\cite{HankePapeSchick:CodimensionTwoIndex}.
Moreover, Kubota~\cite{kubotaRelativeMishchenkoFomenko2019,KubotaSchick:Codim2} showed that if \(X\) is closed, then it has non-vanishing Rosenberg index.

\section[Ahat-Largeness implies infinite KO-width]{\(\boldsymbol{\Ahat}\)-Largeness implies infinite \(\boldsymbol{\mathcal{KO}}\)-width} \label{A-hat-largeness}

In this section, we will observe that \(\Ahat\)-iso-enlargeability implies infinite \(\Ahat\)-overtorical width and thereby prove Theorem~\ref{thm:A-hat-iso-enlargeable}.

The proof is based on analogous ideas as in \cite[Section~3]{Gromov:MetricInequalitiesScalar}:
Observe that each sphere \(\Sphere^{k}\) for \(k \geq 1\) admits an embedding of the \((k-1)\)-torus \(\Torus^{k-1}\).
Extending this to a tubular neighborhood shows that there exists an embedded torical band \(\Torus^{k-1} \times [-1,1] \cong V_k \hookrightarrow \Sphere^{k}\) of some width \(d_k > 0\).\footnote{Some ideas to determine the best possible such \(d_k\) are discussed in \cite[Section~3]{Gromov:MetricInequalitiesScalar} but this is not relevant for the purpose here.}
In the case \(k = 1\), we take \(\Torus^{0}\) to be a point and \(V_1\) is just an interval embedded in the circle.
We will assume that \(V_k\) is chosen such that it does not meet the base-point \(\ast \in \Sphere^k\).
Then we set \(\delta_n = \min\limits_{1 \leq k \leq n} d_k\).
\begin{Theorem}\label{thm:A-hat-iso-enlargeable-technical}
 Let \(M\) be a closed manifold which is \(\Ahat\)-iso-enlargeable.
 Then \(\Ahat(M) \neq 0\) or \(M\) has infinite \(\hat{\mathcal{T}}_{\Ahat}\)-width.

 In particular, if \(M\) is spin, then \(\Ahat(M) \neq 0\) or \(M\) has infinite \(\mathcal{KO}\)-width, and thus \(M\) does not admit a metric of positive scalar curvature.
\end{Theorem}
\begin{proof}
 If \(M\) is \(\Ahat\)-iso-enlargeable in dimension \(k = 0\), then this just means that \(\Ahat(M) \neq 0\).

 Otherwise, for each \(\varepsilon > 0\), there exists a local isometry \(\bar{X} \to X\) for some Riemannian mani\-fold~\(\bar{X}\) and a smooth \(\varepsilon\)-Lipschitz map \(f \colon \bar{X} \to \Sphere^k\) for some \(1 \leq k \leq n\) which is constant at infinity and has non-zero \(\Ahat\)-degree.

 Let \(V_k \hookrightarrow \Sphere^k\) be the embedded torical band discussed above.
 We may assume that \(f\) is transversal to the boundary of \(V_k\) and set \(\hat{V} \coloneqq f^{-1}(V_k)\).
 Then \(\hat{V}\) is a band and \(\hat{V} \in \hat{\mathcal{T}}_{\Ahat}\) because \(V_k \cong \Torus^{k-1} \times [-1,1]\) and \(f\) has non-trivial \(\Ahat\)-degree.
 Moreover, \(\width(V, g) \geq \delta_n / \varepsilon\) because \(V_k\) has width \(\geq \delta_n\) and \(f\) is \(\varepsilon\)-Lipschitz.
 As \(\varepsilon > 0\) can be arbitrarily small, this shows that \(M\) has infinite \(\hat{\mathcal{T}}_{\Ahat}\)-width.

 If \(M\) is spin, then, in the first case, \(\Ahat(M) \neq 0\) is the classical Lichnerowicz obstruction to positive scalar curvature.
 In the other case, Corollary~\ref{cor:KO-width-bound-Ahat} implies that \(M\) has infinite \(\mathcal{KO}\)-width, and so \(M\) does not admit a metric of positive scalar curvature by Corollary~\ref{cor:KO-width-obstr}.
\end{proof}

\begin{Remark}
 Conversely, the approach from \cite[Section~4]{Gromov:MetricInequalitiesScalar} can be used to show that infinite \(\hat{\mathcal{T}}_{\Ahat}\)-width implies \(\Ahat\)-iso-enlargeability in some dimension \(k \geq 1\).
\end{Remark}

\section{Generaliziations and open problems} \label{sec:Problems}
In this section, we discuss a number of possible generalizations and open problems which are suggested by the results discussed in this note.
\subsection{Infinite scalar curvature decay via hypersurface}
Theorem~\ref{thm:quantitative-codim-1} suggests an estimate like \eqref{eq:negativity} could hold whenever the closed hypersurface \(M\) does not admit a metric of positive scalar curvature.
In particular, we may propose the following question:

\begin{Conjecture}
 There exist functions \(s_n \colon [0, 2 \pi/n) \to [0,\infty)\), where \(s_n(x) \to \infty\) as \(x \to 2\pi /n\), such that the following holds.

 Let \(M\) be a closed manifold of dimension \(n-1 \neq 4\) which does not admit a metric of positive scalar curvature.
 Suppose that \(g\) is a complete Riemannian metric on \(M \times \R\) whose scalar curvature is bounded below by \(n(n-1)\) on a neighborhood \(U\) of \(M \times \{0\}\) of some width \(d < 2\pi/n\) $($more precisely, that is, \(U \coloneqq \{p \in M \times \R \mid d_g(p, M \times \{0\}) \leq d/2\})\). Then
 \[
 \inf_{p \in M \times \R} \scal_g(p) \leq - s_n(d).
 \]
\end{Conjecture}
Note that this would imply both Conjectures~\ref{conj:band-width} and~\ref{conj:no-complete-psc}.

\subsection{Almost spin manifolds}
While the Dirac operator method relies on manifolds being spin, it is usually enough if the manifold admits a connected covering that is spin.
We refer to such manifolds as \emph{almost spin}.
For simplicity, we have not stressed this point in this note.
However, we expect that generalizations to almost spin manifolds of our statements, in particular Theorem~\ref{thm:quantitative-codim1-obstr}, can be worked out along the same lines.
In the context of the Rosenberg index, the formal groundwork for this has already been laid by Stolz~\cite{Stolz98Concordance} who introduced a twisted version of the group \textCstar-algebra and the Rosenberg index for almost spin manifolds.

\subsection{Comparison to the coarse index}\label{subsec:coarse}
There is another kind of index theory which has been successfully applied to obtain obstructions to \emph{uniformly} positive scalar curvature on certain non-compact manifolds, namely Roe's \emph{coarse index theory} (for an introduction see, e.g.,~\cite{Roe:IndexTheoryCoarseGeometry}).
Coarse index theory was used by Hanke, Pape,
and Schick~\cite{HankePapeSchick:CodimensionTwoIndex} in their approach to the index-theoretic codimension two obstruction (compare Example~\ref{ex:codim-2-obstr}) and by Chang, Weinberger and Yu~\cite{ChangWeinbergerYu:Taming3} to provide obstructions to complete uniformly positive scalar curvature metrics (compare Proposition~\ref{prop:boundary-obstr}).
For a selection of other related results see for instance~\cite{Chang:CoarseObstrArithmetic,Engel:WrongWay,Engel:Rough,SchickZadeh:MultiPart, Zeidler:IndexObstructionPositive}.

Crucial features of the coarse index are that it provides an obstruction to uniformly positive scalar curvature outside a compact subset, and it admits a way of localizing index information on submanifolds.
In the codimension one case, the latter is facilitated by Roe's \emph{partitioned manifold index theorem} which states that a non-vanishing index of a partitioning hypersurface in a non-compact complete manifold implies non-vanishing of the coarse index.

For instance, this can be used to show that if \(M\) is a spin manifold of non-vanishing Rosenberg index, then \(M \times \R\) does not admit a complete metric which has uniformly positive scalar curvature outside a compact subset.
However, this is \emph{also} a consequence of Theorem~\ref{thm:quantitative-codim1-obstr}.
Indeed, it appears that in all the examples we know, the outright obstructions to complete uniformly positive scalar curvature obtained through the coarse index can be alternatively proved (and strengthened) by Theorem~\ref{thm:quantitative-codim1-obstr}.
On the other hand, coarse index theory can be used to exclude the existence of uniformly positive scalar curvature in a fixed quasi-isometry class of Riemannian metrics (the simplest example being \(\R^n\) for \(n \geq 3\) which admits a complete metric of uniformly positive scalar curvature but not in the same quasi-isometry class as the standard Euclidean metric).

This is a somewhat vague point, but it would be of interest to explore if there are a more concrete relations between coarse index theory and the Callias-type operator methods exhibited here (beyond the fact that they have a partially overlapping range of applications).

\subsection{Implications between notions of largeness}
As mentioned Section~\ref{sec:largeness} in the discussion after~\eqref{eq:largeness-implications}, it is an open problem to understand which of the obvious implications between the largeness properties of hypereuclidean, hyperspherical, enlargeable and iso-enlargeable admit a converse.
In particular, this includes Gromov's question of whether closed iso-enlargeable manifolds are enlargeable.

\subsection[KO-width versus the Rosenberg index]{\(\boldsymbol{\mathcal{KO}}\)-width versus the Rosenberg index}
This is an index-theoretic variant of the previous point.
We have discussed in Section~\ref{sec:KO-width} that infinite \(\mathcal{KO}\)-width is an obstruction to (uniformly) positive scalar curvature.
Since infinite \(\mathcal{KO}\)-width can plausibly be described as an obstruction based on index theory of Dirac operators, we should expect that closed spin manifolds of infinite \(\mathcal{KO}\)-width have non-vanishing Rosenberg index (compare Section~\ref{sec:Rosenberg-index}).

\begin{Conjecture}[{\cite[Conjecture~4.12]{Zeidler:band-width-estimates}}] \label{conj:KO-width}
 Let \(M\) be a closed spin manifold of infinite \(\mathcal{KO}\)-width.
 Then \(\alpha(M) \neq 0 \in \KO_\ast(\Cstar \pi_1 M)\).
\end{Conjecture}

There is further evidence for this conjecture.
Since infinite \(\mathcal{KO}\)-width is stable under products with Bott manifolds (see~\cite[Section~4]{Zeidler:band-width-estimates}), Conjecture~\ref{conj:KO-width} is a consequence of Conjecture~\ref{conj:stable-glr} (and thereby of injectivity of the real Baum--Connes assembly map by Theorem~\ref{thm:Stolz-Stable-GLR}).

Moreover, (\(\Ahat\))-iso-enlargeable spin manifolds have infinite \(\mathcal{KO}\)-width by Theorem~\ref{thm:A-hat-iso-enlargeable-technical}.
In the case of the a priori stronger of notion of (\(\Ahat\))-enlargeability, it is known that this implies non-vanishing of the Rosenberg index~\cite{Hanke-Kotschick-Roe-Schick:Coarse-Top-Enlargeability,HankeSchick:Enlargeable,HankeSchick:EnlargeableInfinite}.
In this light, Conjecture~\ref{conj:KO-width} can be viewed as an index-theoretic shadow of the question of whether all iso-enlargeable manifolds are enlargeable.

We conclude with a diagram of the known implications and open problems related to iso-enlargeability, \(\mathcal{KO}\)-width and the Rosenberg index for a closed spin manifold \(M\):

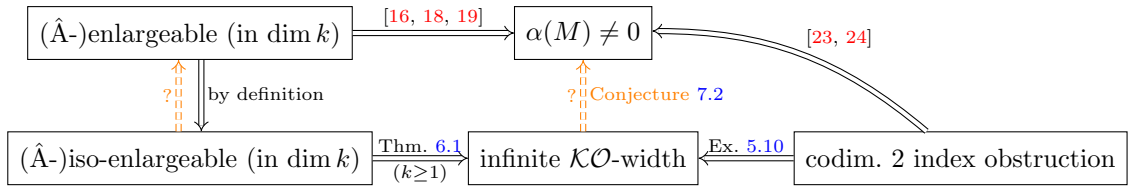
\begin{center}\small
$\begin{tikzcd}[cells={nodes={draw=black, rectangle,anchor=center,minimum height=2em}}, column sep=large, row sep=large]
 \text{(\(\Ahat\)-)enlargeable (in \(\dim k\))} \rar[Rightarrow, "\text{\cite{Hanke-Kotschick-Roe-Schick:Coarse-Top-Enlargeability,HankeSchick:Enlargeable,HankeSchick:EnlargeableInfinite}}"] \dar[Rightarrow, shift left=1ex, "\text{by definition}"] & \alpha(M) \neq 0 \\
 \text{(\(\Ahat\)-)iso-enlargeable (in \(\dim k\))} \rar[Rightarrow, "(k \geq 1)"',"\text{Thm.~\ref{thm:A-hat-iso-enlargeable-technical}}"] \uar[Rightarrow, dashed, orange, shift left=1ex, "?"]
 & \text{infinite \(\mathcal{KO}\)-width} \uar[Rightarrow, dashed, orange, "?", "\text{Conjecture~\ref{conj:KO-width}}"'] & \lar[Rightarrow, "\text{Ex.~\ref{ex:codim-2-obstr}}"']\text{codim.~2 index obstruction} \ular[Rightarrow, "\text{\cite{kubotaRelativeMishchenkoFomenko2019,KubotaSchick:Codim2}}"', bend right=20]
\end{tikzcd}$
\end{center}

\subsubsection*{Funding acknowledgement}
Funded by the Deutsche Forschungsgemeinschaft (DFG, German Research Foundation), Project-ID 427320536 – SFB 1442, as well as under Germany's Excellence Strategy EXC 2044 390685587, Mathematics M\"unster: Dynamics–Geometry–Structure.
Moreover, part of the research pertaining to this article was conducted while the author was employed at the University of G\"ottingen funded through the DFG RTG 2491 Fourier Analysis and Spectral Theory.

\pdfbookmark[1]{References}{ref}
\LastPageEnding

\end{document}